\documentclass[12pt]{amsart}
\usepackage{amscd,amsmath,amsthm,amssymb}
\usepackage{pstcol,pst-plot,pst-3d}
\usepackage{color}
\usepackage{pstricks}
\usepackage{lineno,xcolor}
\DeclareSymbolFont{largesymbol}{OMX}{yhex}{m}{n}
\DeclareMathAccent{\Widehat}{\mathord}{largesymbol}{"62}

\newpsstyle{fatline}{linewidth=1.5pt}
\newpsstyle{fyp}{fillstyle=solid,fillcolor=verylight}
\definecolor{verylight}{gray}{0.97}
\definecolor{light}{gray}{0.9}
\definecolor{medium}{gray}{0.85}
\definecolor{dark}{gray}{0.6}

 \def\G{{\mathcal G}}

 \def\opn#1#2{\def#1{\operatorname{#2}}} 
 \opn\chara{char} \opn\length{\ell} \opn\pd{pd} \opn\rk{rk}
 \opn\projdim{proj\,dim} \opn\injdim{inj\,dim} \opn\rank{rank}
 \opn\depth{depth} \opn\grade{grade} \opn\height{height}
 \opn\embdim{emb\,dim} \opn\codim{codim}
 
 \opn\Tr{Tr} \opn\bigrank{big\,rank}
 \opn\superheight{superheight}\opn\lcm{lcm}
 \opn\trdeg{tr\,deg}
 \opn\reg{reg} \opn\lreg{lreg} \opn\ini{in} \opn\lpd{lpd}
 \opn\size{size} \opn\sdepth{sdepth}
 \opn\link{link}\opn\fdepth{fdepth}\opn\lex{lex}
 \opn\tr{tr}
 \opn\type{type}
 \opn\Borel{Borel}
\opn\cdeg{cdeg}

 \opn\div{div} \opn\Div{Div} \opn\cl{cl} \opn\Cl{Cl}
 \opn\Spec{Spec} \opn\Supp{Supp} \opn\supp{supp} \opn\Sing{Sing}
 \opn\Ass{Ass} \opn\Min{Min}\opn\Mon{Mon}
 \opn\Ann{Ann} \opn\Rad{Rad} \opn\Soc{Soc}
 \opn\Im{Im} \opn\Ker{Ker} \opn\Coker{Coker} \opn\Am{Am}
 \opn\Hom{Hom} \opn\Tor{Tor} \opn\Ext{Ext} \opn\End{End}
 \opn\Aut{Aut} \opn\id{id}
 
 \opn\nat{nat}
 \opn\pff{pf}
 \opn\Pf{Pf} \opn\GL{GL} \opn\SL{SL} \opn\mod{mod} \opn\ord{ord}
 \opn\Gin{Gin} \opn\Hilb{Hilb}\opn\sort{sort}
 \opn\PF{PF}\opn\Ap{Ap}
 \opn\aff{aff} \opn
 \con{conv} \opn\relint{relint} \opn\st{st}
 \opn\lk{lk} \opn\cn{cn} \opn\core{core} \opn\vol{vol}  \opn\inp{inp} \opn\nilpot{nilpot}
 \opn\link{link} \opn\star{star}\opn\lex{lex}\opn\set{set}
 \opn\width{wd}
 \opn\Fr{F}
 \opn\QF{QF}
 \opn\G{G}
 \opn\type{type}\opn\res{res}
 \opn\gr{gr}
 
  \def\cdeg{deg}

 \def\pot#1#2{#1[\kern-0.28ex[#2]\kern-0.28ex]}

 \opn\dirlim{\underrightarrow{\lim}}
 \opn\inivlim{\underleftarrow{\lim}}
 \def\Implies{\ifmmode\Longrightarrow \else
         \unskip${}\Longrightarrow{}$\ignorespaces\fi}
 \def\implies{\ifmmode\Rightarrow \else
         \unskip${}\Rightarrow{}$\ignorespaces\fi}
 \def\iff{\ifmmode\Longleftrightarrow \else
         \unskip${}\Longleftrightarrow{}$\ignorespaces\fi}

 \let\:=\colon
 \newtheorem{Theorem}{Theorem}[section]
 \newtheorem{Lemma}[Theorem]{Lemma}
 \newtheorem{Corollary}[Theorem]{Corollary}

 \newtheorem{Example}[Theorem]{Example}
 
 \newtheorem{Definition}[Theorem]{Definition}

 \let\epsilon\varepsilon
 \let\kappa=\varkappa
 \def\qed{\ifhmode\textqed\fi
       \ifmmode\ifinner\quad\qedsymbol\else\dispqed\fi\fi}
 \def\textqed{\unskip\nobreak\penalty50
        \hskip2em\hbox{}\nobreak\hfil\qedsymbol
        \parfillskip=0pt \finalhyphendemerits=0}
 \def\dispqed{\rlap{\qquad\qedsymbol}}
 
 \opn\dis{dis}
 \def\pnt{{\raise0.5mm\hbox{\large\bf.}}}
 
 \opn\Lex{Lex}

\begin{document}
\title {Projective dimension and  regularity of  powers of edge ideals of vertex-weighted  rooted forests}

\author { Li Xu,  Guangjun Zhu$^{^*}$,  Hong Wang and  Jiaqi Zhang }

\address{Authors¡¯ address:  School of Mathematical Sciences, Soochow
University, Suzhou 215006, P.R. China}
\email{zhuguangjun@suda.edu.cn(Corresponding author:Guangjun Zhu),
\linebreak[4] 1240470845@qq.com(Li Xu), 651634806@qq.com(Hong Wang),\nolinebreak[2] zjq7758258@vip.qq.com(Jiaqi Zhang).}

\dedicatory{ }

\begin{abstract}
In this paper we provide some exact formulas for  projective dimension  and the regularity of powers
of edge ideals of  vertex-weighted   rooted forests. These formulas are  functions of the weight of the vertices and the number of edges.
 We  also give some examples to show that these  formulas are related to direction selection and
the assumptions about ``rooted"  forest such that
$w(x)\geq 2$  if $d(x)\neq 1$ cannot be dropped.
\end{abstract}

\thanks{* Corresponding author}




\maketitle

\setcounter{tocdepth}{1}

\section{Introduction}

\hspace{3mm} Let $S=k[x_{1},\dots,x_{n}]$ be a polynomial ring in $n$ variables over a field $k$ and let $I\subset S$ be a homogeneous ideal.
There are two central invariants associated to $I$, the regularity $\mbox{reg}\,(I):=\mbox{max}\{j-i\ |\ \beta_{i,j}(I)\neq 0\}$
 and the projective dimension $\mbox{pd}\,(I):=\mbox{max}\{i\ |\ \beta_{i,j}(I)\neq 0\ \text{for some}\ j\}$, that in a sense, they measure the complexity of computing the graded Betti numbers $\beta_{i,j}(I)$ of $I$.  In particular, if $I$ is a
monomial ideal,  its polarization  $I^{\mathcal{P}}$ has the same projective dimension and regularity as $I$ and is squarefree. Thus one can associate $I^{\mathcal{P}}$ to a graph or a hypergraph or  a simplicial complex.
  Many authors have studied the regularity
and Betti numbers of edge ideals of graphs, e.g. \cite{AF1,AF2,BHO,EK,HT2,HT3,J,KM,MV,W,Z1,Z2,Z3}. Other authors
have studied higher degree generalizations using hypergraphs and clutters \cite{DHS,DS1,DS2,DS3,LM}
or simplicial complexes \cite{EF,F1}.

Given  any homogeneous ideal $I$, it is well known that the regularity of $I^t$ is asymptotically a linear function in $t$, that is, there exist constants $a$ and $b$ such that for all $t\gg 0$, $\mbox{reg}\,(I^t)=at+b$ (see \cite{Ch,CHT,K,TW}). Generally, the problem of finding the exact linear form $at+b$ and the
smallest value $t_0$ such that $\mbox{reg}\,(I^t)=at+b$ for all $t\geq t_ 0$ has proved to be very difficult. In \cite{B2}, Brodmann  showed
that  $\mbox{depth}\,(S/I^t)$ is a constant for $t\gg 0$, and this constant is bounded
above by $n-\ell(I)$, where $\ell(I)$ is the analytic
spread of $I$.  It is shown  in \cite[Theorem 1.2]{HH3} that $\mbox{depth}\,(S/I^t)$ is a nonincreasing
function of $t$ when all powers of $I$ have a linear resolution
and conditions are given in that paper under which all powers of $I$ will
have linear quotients.
By Auslander-Buchsbaum formula, we obtain the projective dimension $\mbox{pd}\,(S/I^t)$ is a constant for $t\gg 0$.
In this regard, there has been an interest in determining the smallest value $t_0$ such that $\mbox{pd}\,(S/I^t)$ is a constant for all $t\geq t_ 0$.
 (see  \cite{E,EF,FM,HH3,JNS,M}).

To the best of our knowledge, a few papers consider how to compute  $\mbox{pd}\,(S/I^t)$ for a homogeneous ideal $I$.

A {\em directed graph} or {\em digraph} $D$ consists of a finite set $V(D)$ of vertices, together
with a collection $E(D)$ of ordered pairs of distinct points called edges or
arrows. A vertex-weighted directed graph is a triplet $D=(V(D), E(D),w)$, where  $w$ is a weight function $w: V(D)\rightarrow \mathbb{N}^{+}$, where $N^{+}=\{1,2,\ldots\}$.
Some times for short we denote the vertex set $V(D)$ and edge set $E(D)$
by $V$ and $E$ respectively.
The weight of $x_i\in V$ is $w(x_i)$, denoted by $w_i$ or $w_{x_i}$.

The edge ideal of a vertex-weighted digraph was first introduced by Gimenez et al \cite{GBSVV}. Let $D=(V,E,w)$ be a  vertex-weighted digraph with the vertex set $V=\{x_{1},\ldots,x_{n}\}$. We consider the polynomial ring $S=k[x_{1},\dots, x_{n}]$ in $n$ variables over a field $k$. The edge ideal of $D$,  denoted by $I(D)$, is the ideal of $S$ given by
$$I(D)=(x_ix_j^{w_j}\mid  x_ix_j\in E).$$

Edge ideals of weighted digraphs arose in the theory of Reed-Muller codes as initial ideals of vanishing ideals
of projective spaces over finite fields \cite{MPV,PS}.
If a vertex $x_i$ of $D$ is a source (i.e., has only arrows leaving $x_i$) we shall always
assume $w_i=1$ because in this case the definition of $I(D)$ does not depend on the
weight of $x_i$.  If  $w_j=1$ for all $j$, then $I(D)$ is the edge ideal of underlying graph $G(D)$ of $D$.
Many researchers have tried to compute  $a,b$ and $t_ 0$  such that $\mbox{reg}\,(I(G(D))^t)=at+b$ for all $t\geq t_ 0$
for the edge ideals of some special families of graphs (see \cite{AB,ABS,B1,BBH1,BBH2,BHT,MSY,RJNP}).
Some authors have studied  bounds for $\text{depth}\,(S/I(G(D))^t)$ or $\text{pd}\,(I(G(D))^t)$
for    edge ideals of some special  graphs (see \cite{HT3,HH2,FM,M}). In \cite{Z4}, the first three authors  derive some exact formulas for the projective
dimension and regularity of the edge ideals of vertex-weighted rooted forests and oriented cycles.
To the best of our knowledge, there is no result about the projective
dimension  and the regularity of $I(D)^t$ for a  vertex-weighted digraph.

In this article, we are interested in algebraic properties corresponding to the projective
dimension and the regularity of $I(D)^t$ for some  vertex-weighted oriented graphs. By using the approaches of Betti splitting and polarization, we derive some exact formulas  for the projective
dimension  and the regularity  of powers of edge ideals of some directed graphs.
The results are as follows:

\begin{Theorem}\label{thm5}
Let $D=(V(D),E(D),w)$ be a vertex-weighted rooted forest such that  $w(x)\geq 2$  if $d(x)\neq 1$. Then
$$\mbox{reg}\,(I(D)^{t})=\sum\limits_{x\in V(D)}w(x)-|E(D)|+1+(t-1)(w+1)\ \ \  \mbox{for all}\ \ t\geq 1,$$
where $w=\mbox{max}\,\{w(x)\mid x \in V(D)\}$.
\end{Theorem}

\begin{Theorem}\label{thm6}
Let $D=(V(D),E(D),w)$ be a vertex-weighted rooted forest such that    $w(x)\geq 2$  if $d(x)\neq 1$. Then
$$\mbox{pd}\,(I(D)^{t})=|E(D)|-1 \ \  \mbox{for all}\ \ t\geq 1.$$
\end{Theorem}

Our paper is organized as follows. In section $2$, we recall some
definitions and basic facts used in the following sections. In section $3$, we
provide some exact formulas for the regularity of the powers of the edge
ideals of vertex-weighted line graphs. Meanwhile, we give some examples to show the regularity of the powers of
the edge ideals of vertex-weighted oriented line graphs are related to
direction selection and  the assumption  that
$w(x)\geq 2$  if $d(x)\neq 1$ cannot be dropped.
 In section $4$,   we give some exact formulas for the projective
dimension and the regularity of the powers of the edge
ideals of vertex-weighted  rooted forests. Moreover, we also give some examples to show the projective
dimension and the regularity of the powers of
the edge ideals of vertex-weighted oriented rooted forests are related to
direction selection and  the assumption  that
$w(x)\geq 2$  if $d(x)\neq 1$ cannot be dropped.

\medskip
For all unexplained terminology and additional information, we refer to \cite{JG} (for the theory
of digraphs), \cite{BM} (for graph theory), and \cite{BH,HH2} (for the theory of edge ideals of graphs and
monomial ideals).  We greatfully acknowledge the use of the computer algebra system CoCoA (\cite{Co}) for our experiments.

\medskip

\section{Preliminaries }

In this section, we gather together the needed  definitions and basic facts, which will
be used throughout this paper. However, for more details, we refer the reader to \cite{AF2,BBH1,BM,E,FHT,HT1,HH2,J,JG,MPV,PRT,Z2,Z4}.

A {\em directed graph} or {\em digraph} $D$ consists of a finite set $V(D)$ of vertices, together
with a collection $E(D)$ of ordered pairs of distinct points called edges or
arrows. If $\{u,v\}\in E(D)$ is an edge, we write $uv$ for $\{u,v\}$, which is denoted to be the directed edge
where the direction is from $u$ to $v$ and $u$ (resp. $v$) is called the {\em starting}  point (resp. the {\em ending} point).
Given any digraph $D$, we can associate a graph $G$ on the same vertex set
simply by replacing each arrow by an edge with the same ends. This graph is called the
underlying graph of $D$, denoted by $G(D)$. Conversely, any graph $G$ can be regarded as
a digraph, by replacing each of its edges by just one of the  two oppositely oriented arrows with the same ends.
 Such a digraph is called an {\em  orientation} of $G$.
An orientation of a simple graph is referred to as an {\em simple  oriented}  graph.

Every concept that is valid for graphs automatically applies to digraphs too.
For example, the degree of a vertex $x$ in a digraph $D$, denoted $d(x)$, is simply the degree of $x$ in
$G(D)$. Likewise, a digraph is said to be connected if
its underlying graph is connected. An {\em  oriented}  path or {\em  oriented}  cycle  is an orientation of a
path or cycle in which each vertex dominates its successor in the sequence.
An {\em  oriented} acyclic graph is a simple digraph without oriented cycles.
An {\em  oriented tree} or {\em polytree} is a  oriented acyclic graph formed by orienting the edges of undirected acyclic graphs.
A  {\em rooted tree} is an oriented tree in which all edges  are oriented  either away from or
towards the root.  Unless specifically stated, a rooted tree in this article
 is an oriented tree in which all edges  are oriented away from  the root.
An {\em oriented forest} is a disjoint union of oriented trees. A {\em rooted forest} is a disjoint union of rooted trees.

 A vertex-weighted oriented graph is a triplet $D=(V(D), E(D),w)$, where $V(D)$ is the  vertex set,
$E(D)$ is the edge set and $w$ is a weight function $w: V(D)\rightarrow \mathbb{N}^{+}$, where $N^{+}=\{1,2,\ldots\}$.
Some times for short we denote the vertex set $V(D)$ and edge set $E(D)$
by $V$ and $E$ respectively.
The weight of $x_i\in V$ is $w(x_i)$, denoted by $w_i$  or $w_{x_i}$..
Given  a  vertex-weighted oriented graph $D=(V,E,w)$ with the vertex set $V=\{x_{1},\ldots,x_{n}\}$, we may consider the polynomial ring $S=k[x_{1},\dots, x_{n}]$ in $n$ variables over a field $k$. The edge ideal of $D$,  denoted by $I(D)$, is the ideal of $S$ given by
$$I(D)=(x_ix_j^{w_j}\mid  x_ix_j\in E).$$

If a vertex $x_i$ of $D$ is a source (i.e., has only arrows leaving $x_i$) we shall always
assume $w_i=1$ because in this case the definition of $I(D)$ does not depend on the
weight of $x_i$.

\medskip
For any homogeneous ideal $I$ of the polynomial ring  $S=k[x_{1},\dots,x_{n}]$, there exists a {\em graded
minimal finite free resolution}

\vspace{3mm}
$0\rightarrow \bigoplus\limits_{j}S(-j)^{\beta_{p,j}(I)}\rightarrow \bigoplus\limits_{j}S(-j)^{\beta_{p-1,j}(I)}\rightarrow \cdots\rightarrow \bigoplus\limits_{j}S(-j)^{\beta_{0,j}(I)}\rightarrow I\rightarrow 0,$
where the maps are exact, $p\leq n$, and $S(-j)$ is the $S$-module obtained by shifting
the degrees of $S$ by $j$. The number
$\beta_{i,j}(I)$, the $(i,j)$-th graded Betti number of $I$, is
an invariant of $I$ that equals the number of minimal generators of degree $j$ in the
$i$th syzygy module of $I$.
Of particular interests are the following invariants which measure the ¡°size¡± of the minimal graded
free resolution of $I$.
The projective dimension of $I$, denoted pd\,$(I)$, is defined to be
$$\mbox{pd}\,(I):=\mbox{max}\,\{i\ |\ \beta_{i,j}(I)\neq 0\}.$$
The regularity of $I$, denoted $\mbox{reg}\,(I)$, is defined by
$$\mbox{reg}\,(I):=\mbox{max}\,\{j-i\ |\ \beta_{i,j}(I)\neq 0\}.$$

\vspace{3mm}We now derive some formulas for $\mbox{pd}\,(I)$  and $\mbox{reg}\,(I)$ in some special cases by using some
tools developed in \cite{FHT}.

\begin{Definition} \label{bettispliting}Let $I$  be a monomial ideal, and suppose that there exist  monomial
ideals $J$ and $K$ such that $\mathcal{G}(I)$ is the disjoint union of $\mathcal{G}(J)$ and $\mathcal{G}(K)$, where $\mathcal{G}(I)$ denotes the unique minimal set of monomial generators of $I$. Then $I=J+K$
is a {\em Betti splitting} if
$$\beta_{i,j}(I)=\beta_{i,j}(J)+\beta_{i,j}(K)+\beta_{i-1,j}(J\cap K)\hspace{2mm}\mbox{for all}\hspace{2mm}i,j\geq 0,$$
where $\beta_{i-1,j}(J\cap K)=0\hspace{2mm}  \mbox{if}\hspace{2mm} i=0$.
\end{Definition}

\vspace{3mm}This formula was first obtained for the total Betti numbers by
Eliahou and Kervaire \cite{EK} and extended to the graded case by Fatabbi \cite{F2}.
In  \cite{FHT}, the authors describe some sufficient conditions for an
ideal $I$ to have a Betti splitting. We need  the following lemma.

\begin{Lemma}\label{lem1}(\cite[Corollary 2.7]{FHT}).
Suppose that $I=J+K$ where $\mathcal{G}(J)$ contains all
the generators of $I$ divisible by some variable $x_{i}$ and $\mathcal{G}(K)$ is a nonempty set containing
the remaining generators of $I$. If $J$ has a linear resolution, then $I=J+K$ is a Betti
splitting. Hence
$$\mbox{reg}\,(I)=\mbox{max}\{\mbox{reg}\,(J),\mbox{reg}\,(K),\mbox{reg}\,(J\cap K)-1\}.$$
\end{Lemma}

\medskip
When $I$ is a Betti
splitting ideal, Definition \ref{bettispliting} implies the following results:
\begin{Corollary} \label{cor1}
If $I=J+K$ is a Betti splitting ideal, then
\begin{itemize}
 \item[(1)]$\mbox{reg}\,(I)=\mbox{max}\{\mbox{reg}\,(J),\mbox{reg}\,(K),\mbox{reg}\,(J\cap K)-1\}$,
 \item[(2)] $\mbox{pd}\,(I)=\mbox{max}\{\mbox{pd}\,(J),\mbox{pd}\,(K),\mbox{pd}\,(J\cap K)+1\}$.
\end{itemize}
\end{Corollary}

\medskip
We need the following Lemma:
\begin{Lemma}
\label{lem2}(\cite[Lemma 2.2 and  Lemma 3.2 ]{HT1})
Let $S_{1}=k[x_{1},\dots,x_{m}]$, $S_{2}=k[x_{m+1},\dots,x_{n}]$ and $S=k[x_{1},\dots,x_{n}]$ be three polynomial rings, $I\subseteq S_{1}$ and
$J\subseteq S_{2}$ be two proper non-zero homogeneous  ideals.  Then
\begin{itemize}
\item[(1)] $\mbox{reg}\,(S/(I+J))=\mbox{reg}\,(S_{1}/I)+\mbox{reg}\,(S_{2}/J)$,
\item[(2)] $\mbox{pd}\,(S/(I+J))=\mbox{pd}\,(S_{1}/I)+\mbox{pd}\,(S_{2}/J)$.
\end{itemize}
\end{Lemma}

\medskip
  Let $\mathcal{G}(I)$ denote the minimal set of generators of a monomial ideal $I\subset S$
 and let $u\in S$ be a monomial, we set $\mbox{supp}(u)=\{x_i: x_i|u\}$. If $\mathcal{G}(I)=\{u_1,\ldots,u_m\}$, we set $\mbox{supp}(I)=\bigcup\limits_{i=1}^{m}\mbox{supp}(u_i)$. The following lemma is well known.
 \begin{Lemma}
\label{lem3} Let  $I, J=(u)$ be two monomial ideals  such that $\mbox{supp}\,(u)\cap \mbox{supp}\,(I)=\emptyset$. If the degree of $u$  is  $d$. Then
\begin{itemize}
\item[(1)] $\mbox{reg}\,(J)=d$,
\item[(2)]$\mbox{reg}\,(JI)=\mbox{reg}\,(I)+d$.
\end{itemize}
\end{Lemma}

\medskip
\begin{Definition} \label{polarization}
Suppose that $u=x_1^{a_1}\cdots x_n^{a_n}$ is a monomial in $S$. We define the {\it polarization} of $u$ to be
the squarefree monomial $$\mathcal{P}(u)=x_{11}x_{12}\cdots x_{1a_1} x_{21}\cdots x_{2a_2}\cdots x_{n1}\cdots x_{na_n}$$
in the polynomial ring $S^{\mathcal{P}}=k[x_{ij}\mid 1\leq i\leq n, 1\leq j\leq a_i]$.
If $I\subset S$ is a monomial ideal with $\mathcal{G}(I)=\{u_1,\ldots,u_m\}$,  the  {\it polarization}
of $I$,  denoted by $I^{\mathcal{P}}$, is defined as:
$$I^{\mathcal{P}}=(\mathcal{P}(u_1),\ldots,\mathcal{P}(u_m)),$$
which is a squarefree monomial ideal in the polynomial ring $S^{\mathcal{P}}$.
\end{Definition}

\medskip
Here is an example of how polarization works.

\begin{Example} \label{example1} Let $I(D)=(x_1x_2^{3},x_2x_3,x_3x_4^{2},x_4x_5^{5})$ be the edge ideal of a vertex-weighted rooted tree $D$, then the polarization of $I(D)$ is the ideal  $I(D)^{\mathcal{P}}=(x_{11}x_{21}x_{22}x_{23},\\
x_{21}x_{31},x_{31}x_{41}x_{42},x_{41}x_{51}x_{52}x_{53}x_{54}x_{55})$.
\end{Example}

\medskip
A monomial ideal $I$ and its polarization $I^{\mathcal{P}}$ share many homological and
algebraic properties.  The following is a very useful property of polarization.

\begin{Lemma}
\label{lem4}(\cite[Corollary 1.6.3]{HH2}) Let $I\subset S$ be a monomial ideal and $I^{\mathcal{P}}\subset S^{\mathcal{P}}$ its polarization.
Then
\begin{itemize}
\item[(1)] $\beta_{ij}(I)=\beta_{ij}(I^{\mathcal{P}})$ for all $i$ and $j$,
\item[(2)] $\mbox{reg}\,(I)=\mbox{reg}\,(I^{\mathcal{P}})$,
\item[(2)] $\mbox{pd}\,(I)=\mbox{pd}\,(I^{\mathcal{P}})$.
\end{itemize}
\end{Lemma}

\medskip
The following lemma can be used for computing the projective dimension and the regularity of  an  ideal.
\begin{Lemma}
\label{lem5}(\cite[Lemma 1.1 and Lemma 1.2]{HTT})  Let\ \ $0\rightarrow A \rightarrow  B \rightarrow  C \rightarrow 0$\ \  be a short exact sequence of finitely generated graded $S$-modules.
Then
\begin{itemize}
\item[(1)] $\mbox{pd}\,(B)= \mbox{pd}\,(A)$ if $\mbox{pd}\,(A)\geq \mbox{pd}\,(C)$,
\item[(2)]$\mbox{reg}\,(B)\leq \mbox{max}\, \{\mbox{reg}\,(A),\mbox{reg}\,(C)\}$,
\item[(3)]$\mbox{reg}\,(B)= \mbox{reg}\,(C)$ if $\mbox{reg}\,(C)\geq \mbox{reg}\,(A)$,
\item[(4)]$\mbox{reg}\,(B)= \mbox{reg}\,(A)$ if $\mbox{reg}\,(A)>\mbox{reg}\,(C)+1$,
\item[(5)]$\mbox{reg}\,(C)= \mbox{reg}\,(A)-1$ if $\mbox{reg}\,(A)> \mbox{reg}\,(B)$,
\item[(6)]$\mbox{reg}\,(C)= \mbox{reg}\,(B)$ if $\mbox{reg}\,(A)< \mbox{reg}\,(B)$.
\end{itemize}
\end{Lemma}

\vspace{5mm}
\section{Regularity of  powers of edge ideals of vertex-weighted oriented  line graphs}

\vspace{5mm}In this section, by using the approaches of Betti splitting and polarization, we will provide some formulas for the regularity of the powers of the edge ideals of  vertex-weighted oriented  line graphs.  We also give some examples to show the assumptions that
  $w(x)\geq 2$  if $d(x)\neq 1$ in  oriented
line graph cannot be dropped.
We shall start with the following two lemmas.

\begin{Lemma}
\label{lem6}(\cite[Lemma 1.3]{HTT}) Let $R$ be the polynomial ring over a field and let $I$ be a proper homogeneous
ideal in $R$. Then
\begin{itemize}
\item[(1)] $\mbox{pd}\,(I)=\mbox{pd}\,(R/I)-1$,
\item[(2)] $\mbox{reg}\, (I)=\mbox{reg}\,(R/I)+1$.
\end{itemize}
\end{Lemma}

\medskip
\begin{Lemma}
\label{lem7}Let $R=k[x_1,\ldots,x_m]$ be the polynomial ring over a field $k$ and  $I$  a proper homogeneous
ideal in $R$. Let $y$ be another variable and $S=R[y]$. Then
\begin{itemize}
\item[(1)]$\mbox{pd}\,(S/IS)=\mbox{pd}\,(R/IR)+1$,
\item[(2)] $\mbox{reg}\, (S/IS)=\mbox{reg}\,(R/IR)$.
\end{itemize}
\end{Lemma}
\begin{proof}The results follow from
$S/IS\cong (R/IR)[y]\cong  R/IR\otimes_{k} k[y]$.
\end{proof}

\medskip
The following theorem generalizes Lemma 4.4 of \cite{BHT}.
\begin{Theorem}
\label{thm1} Let $u_1,u_2,\ldots,u_{r}$  be a regular sequence of homogeneous polynomials in $S$ with
$\cdeg\,(u_i)=d_i$ for $i=1,\ldots,r$. Let $I=(u_1,u_2,\ldots,u_{r})$ be an ideal. Then
$$\mbox{reg}\,(I^t)=\sum\limits_{i=1}^{r}d_i-(r-1)+(t-1)w \hspace{1cm}  \mbox{for all}\ \ t\geq 1,$$
where $w=max\{d_i \mid  1\leq i\leq r \}$.
\end{Theorem}
\begin{proof}  We apply induction on  $r$ and $t$. For $r=1$, the statements are clear for  all $t\geq 1$. If $r\geq 2$, it is obvious that $\mbox{reg}\,(I)=\sum\limits_{i=1}^{r}d_i-(r-1)$ by Lemmas \ref{lem2} and \ref{lem6}.  Therefore, we may
suppose that $r,t\geq 2$.
 For convenience, let's assume that $w=d_r$ and
 $J=(u_1,\ldots,u_{r-1})$. Then
 $$I^t=(J+(u_{r}))I^{t-1}=J^t+u_rI^{t-1}.$$
Hence there exists a surjection:
$$\phi :\ I^{t-1}(-w)\oplus J^t \overset{\cdot (u_r,1)}\longrightarrow  I^t.$$
Since $u_r$ is a regular  element of $S/J$, the kernel of $\phi$ is  $u_rJ^t$.
Thus we  have the  short exact sequence
$$0\rightarrow J^t(-w) \rightarrow I^{t-1}(-w)\oplus J^t\stackrel{\cdot (u_r,1)} \longrightarrow I^t \rightarrow 0.$$
By induction hypotheses on $t$ and $r$, we obtain that
\begin{eqnarray*}\mbox{reg}\,(I^{t-1}(-w))&=&\mbox{reg}\,(I^{t-1})+w=\sum\limits_{i=1}^{r}d_i-(r-1)+(t-2)w+w\\
&=&\sum\limits_{i=1}^{r}d_i-(r-1)+(t-1)w\\
\end{eqnarray*}
and
\[\mbox{reg}\,(J^t)=\sum\limits_{i=1}^{r-1}d_i-(r-2)+(t-1)w'
\]
where $w'=\mbox{max}\,\{d_i \mid  1\leq i\leq r-1\}$.

Notice that $w\geq w'$, this implies $\mbox{reg}\,(I^{t-1}(-w))\geq \mbox{reg}\,(J^t)$. It follows that
$$\mbox{reg}\,(I^{t-1}(-w)\oplus J^t)=\mbox{max}\,\{\mbox{reg}\,(I^{t-1}(-w)),\mbox{reg}\,(J^t)\}=\sum\limits_{i=1}^{r}d_i-(r-1)+(t-1)w.$$
Using the fact  $u_rI^{t-1}\cap J^t=u_rJ^t$,  Lemma
\ref{lem3} (2) and  induction hypothesis on $r$, we have
$$\mbox{reg}\,(u_rI^{t-1}\cap J^t)=\mbox{reg}\,(u_rJ^t)=\sum\limits_{i=1}^{r}d_i-(r-2)+(t-1)w'.$$
Let $\alpha=\mbox{reg}\,(I^{t-1}(-w)\oplus J^t)$ and  $\beta=\mbox{reg}\,(u_rI^{t-1}\cap J^t)$, then
$$\beta-\alpha=(t-1)(w'-w)+1.$$
If $w'=w$, then $\beta-\alpha=1$. By Lemma \ref{lem5} (5), we have
$$\mbox{reg}\,(I^t)=\beta-1=\alpha=\sum\limits_{i=1}^{r}d_i-(r-1)+(t-1)w. $$
If $w'\leq w-2$, or $w'=w-1$ and $t\geq 3$, then $\beta<\alpha$. Thus by Lemma \ref{lem5} (6), we have
$$\mbox{reg}\,(I^t)=\alpha=\sum\limits_{i=1}^{r}d_i-(r-1)+(t-1)w. $$
Therefore it is enough to prove this conclusion only for $t=2$ and $w=w'+1$. In this case, $w\geq 2$.
Set $d_{r-1}=w'$ and  $K=(u_1,u_2,\ldots,u_{r-2},u_r)$, we have $I^2=(K+(u_{r-1}))I=u_{r-1}I+K^2$. Thus there exists the  exact sequence
$$0\rightarrow u_{r-1}I\cap K^2 \rightarrow I(-w')\oplus K^2 \stackrel{\cdot (u_{r-1},1)} \longrightarrow I^2 \rightarrow 0.$$
  Using the fact  $w=w'+1$, $u_{r-1}I\cap K^2 =u_{r-1}K^2$ and induction hypothesis, we have
$$\mbox{reg}\,(I(-w'))=w'+\sum\limits_{i=1}^{r}d_i-(r-1)=\sum\limits_{i=1}^{r}d_i-(r-1)+w-1,$$
$$\mbox{reg}\,(K^2)=\sum\limits_{i=1}^{r-2}d_i+d_{r}-(r-2)+w=\sum\limits_{i=1}^{r}d_i-(r-1)+2,$$
and
$$\mbox{reg}\,(u_{r-1}K^2)=w'+\sum\limits_{i=1}^{r}d_i-(r-1)+2=\sum\limits_{i=1}^{r}d_i-(r-1)+w+1.$$
It follows that $\mbox{reg}\,(u_{r-1}K^2)>\mbox{max}\,\{\mbox{reg}(I(-w')),\mbox{reg}\,(K^2)\}$ because of $w\geq 2$, thus the result follows from Lemma \ref{lem5} (5).
\end{proof}

\medskip
 Let  $D=(V(D), E(D),w)$ be a vertex-weighted oriented graph. For $T\subset V(D)$, we define
the {\em induced vertex-weighted  subgraph} $H=(V(H), E(H),w)$ of $D$  to be the vertex-weighted oriented graph
such that $V(H)=T$, $uv\in E(H)$  if and only if $uv\in E(D)$  and   for any $u,v\in V(H)$. For any $u\in V(H)$ and $u$ is not a source in $H$, its weight in $H$ equals to the weight of $u$ in $D$, otherwise, its weight in $H$ equals to $1$.
For  $P\subset V(D)$, we  denote
$D\setminus P$ the induced subgraph of $D$ obtained by removing the vertices in $P$ and the
edges incident to these vertices. If $P=\{x\}$ consists of a single element, then we write $D\setminus x$ for $D\setminus \{x\}$.
For $W\subseteq E(D)$, we define $D\setminus W$ to be the subgraph of $D$ with all edges  in $W$  deleted (but its vertices remained).  When $W=\{e\}$ consists of a single edge, we write $D\setminus e$ instead of $D\setminus \{e\}$.

Let $x\in V(D)$, then we call  $N_D^{+}(x)=\{y:(x,y)\in E(D)\}$ and $N_D^{-}(x)=\{y:(y,x)\in E(D)\}$ to be  the out-neighbourhood and  in-neighbourhood of $x$, respectively. The neighbourhood of $x$ is the set $N_D(x)=N_D^{+}(x)\cup N_D^{-}(x)$.

\medskip
The following lemmas are needed to facilitate calculating the projective dimension and the regularity
of powers through induction on the power.
\begin{Lemma}
\label{lem8}  Let $D=(V(D),E(D),w)$ be a vertex-weighted oriented graph, let $z$ be a leaf with $N_D^{-}(z)=\{y\}$. Then
$$(I(D)^{t},z^{w_z})=(I(D\setminus z)^{t},z^{w_z})\ \ \text{for any}\ \ t\geq 1.$$
\end{Lemma}
\begin{proof} It's clear that  $(I(D\setminus z)^{t},z^{w_z})\subseteq (I(D)^{t},z^{w_z})$.
If any monomial $f\in \mathcal{G}(I(D)^{t})\setminus \mathcal{G}(I(D\setminus z)^{t})$, then $z$  divides  $f$. It follows that
 $yz^{w_z}$  also divides   $f$ because of $N_D^{-}(z)=\{y\}$.  This implies  $f\in (z^{w_z})$.
\end{proof}

\medskip
\begin{Lemma}
\label{lem9} Let $D=(V(D),E(D),w)$ be a vertex-weighted oriented graph as in Lemma
\ref{lem8} such that $z$ is a leaf with $N_D^{-}(z)=\{y\}$. Then
$$(I(D)^{t}:yz^{w_z})=I(D)^{t-1}\ \ \text{for any}\ \ t\geq 2.$$
\end{Lemma}
\begin{proof} If any monomial $f\in \mathcal{G}(I(D)^{t}:yz^{w_z})$, then $fyz^{w_z}\in I(D)^{t}$. We can write
$fyz^{w_z}=e_{i1}e_{i2}\ldots e_{it}h$  for some monomial $h$, where $e_{ij}=x_{ij}y_{ij}^{w(y_{ij})}$ such that $x_{ij}y_{ij}\in E(D)$.
If there exists some  $j\in \{1,\ldots,t\}$ such that $z^{w_z}$  divides  $e_{ij}$, then $e_{ij}=yz^{w_z}$
 because of $N_D^{-}(z)=\{y\}$.
This implies that $f\in I(D)^{t-1}$. If $z^{w_z}$ does not divide   $e_{ij}$ for all $j\in \{1,\ldots,t\}$, then  $z^{w_z}$  divides  $h$. Thus
$f\in I(D)^{t-1}$.
\end{proof}

\medskip
\begin{Lemma}
\label{lem10}  Let $D=(V(D),E(D),w)$ be a vertex-weighted oriented graph as in Lemma
\ref{lem8} such that $z$ is a leaf with $N_D^{-}(z)=\{y\}$. Then
$$((I(D)^{t}:z^{w_z}),y)=((I(D\setminus y)^{t}:z^{w_z}),y)=(I(D\setminus y)^{t},y)\ \ \text{for any}\ \ t\geq 2.$$
\end{Lemma}
\begin{proof}  It's obvious that $((I(D\setminus y)^{t}:z^{w_z}),y)=(I(D\setminus y)^{t},y)$ by similar arguments as the proof of Lemma
\ref{lem9}.
Now assume that  $f\in \mathcal{G}((I(D)^{t}:z^{w_z}),y)$ and $y$ does not divide  $f$, then
$f\in (I(D)^{t}:z^{w_z})$. It follows that
$fz^{w_z}\in I(D\setminus y)^{t}$. This means that
 $f\in (I(D\setminus y)^{t}:z^{w_z})$. Hence $f\in ((I(D\setminus y)^{t}:z^{w_z}),y)$.
\end{proof}

\medskip
\begin{Lemma}\label{lem11}
Let $n\geq 3$ be an integer and $P_n$  a vertex-weighted oriented line graph with edge set $E(P_n)=\{x_1x_2,x_2x_3,\ldots,x_{n-1}x_n\}$, its edge ideal  $I(P_n)=(x_1x_2^{w_2},x_2x_3^{w_3},\\
\ldots,x_{n-1}x_n^{w_n})$ such that  $w_i\geq 2$  for any $2\leq i\leq n-1$. Let $I_n$ be an ideal with the generator set $\mathcal{G}(I_n)=\mathcal{G}(I(P_n)^2)\setminus \{x_1^2x_2^{2w_2},x_2^2x_3^{2w_3},\ldots,x_{n-1}^2x_n^{2w_n}\}$.
Then
$$\mbox{reg}\,(I_n)\leq \sum\limits_{i=1}^{n}w_i-(n-1)+1+(w+1)$$
where $w=\mbox{max}\,\{w_i\mid 1\leq i \leq n\}$.
\end{Lemma}
\begin{proof}
We apply induction on $n$. The case $n=3$ is clear.
Assume that $n\geq 4$. Consider the short exact sequences
 $$0\longrightarrow \frac{S}{(I_n:x_{n}^{w_n})}(-w_n)\stackrel{ \cdot x_{n}^{w_n}} \longrightarrow \frac{S}{I_n}\longrightarrow \frac{S}{(I_n,x_{n}^{w_n})}\longrightarrow 0 \eqno(1)$$
and
$$0\longrightarrow \frac{S}{((I_n:x_{n}^{w_n}):x_{n-1})}(-1)\stackrel{\cdot x_{n-1}} \longrightarrow \frac{S}{(I_n:x_{n}^{w_n})}\longrightarrow \frac{S}{((I_n:x_{n}^{w_n}),x_{n-1})} \longrightarrow 0.\eqno(2)$$
Following the same arguments as Lemmas \ref{lem8} $\sim$ \ref{lem10}, we have  $(I_n,x_n^{w_n})=(I_{n-1},x_n^{w_n})$, $((I_n:x_{n}^{w_n}),x_{n-1})=(I_{n-2},x_{n-1})$ and $(I_n:x_{n-1}x_{n}^{w_n})=I(P_{n-1}).$

By induction hypothesis on $n$, Lemma \ref{lem2} (1), Lemma \ref{lem6} and  \cite[Theorem 3.5]{Z3}, we obtain
\begin{eqnarray*}
\mbox{reg}\,((I_n,x_{n}^{w_n}))\!\!&=&\!\!\mbox{reg}\,((I_{n-1},x_n^{w_n}))=\mbox{reg}\,(I_{n-1})+\mbox{reg}\,((x_n^{w_n}))-1\\
&=&\mbox{reg}\,(I_{n-1})+w_n-1\\
&\leq&\sum\limits_{i=1}^{n-1}w_i-(n-2)+1+w'+1+w_n-1\\
&\leq&\sum\limits_{i=1}^{n}w_i-(n-1)+1+w+1,\hspace{5.8cm} (3)
\end{eqnarray*}
where the last inequality holds because of $w'\leq w$, where  $w'=\mbox{max}\,\{w_i\mid 1\leq i \leq n-1\}$,
$$
\mbox{reg}\,(I_n:x_{n-1}x_{n}^{w_n})=\mbox{reg}\,(I(P_{n-1}))=\sum\limits_{i=1}^{n-1}w_i-(n-2)+1, \eqno(4)
$$
and
\begin{eqnarray*}
\mbox{reg}\,((I_n:x_{n}^{w_n}),x_{n-1})&=&\mbox{reg}\,(I_{n-2},x_{n-1})=\mbox{reg}\,(I_{n-2})\\
&\leq&\sum\limits_{i=1}^{n-2}\!w_i\!-(n-3)+1+w''+1\\
&\leq&\sum\limits_{i=1}^{n-1}\!w_i\!-(n-1)+1+w+1 \hspace{4.7cm} (5)
\end{eqnarray*}
where the last inequality holds because of $w_{n-1}\geq 2$ and $w''\leq w$, here  $w''=\mbox{max}\,\{w_i\mid 1\leq i \leq n-2\}$.

By  Lemma \ref{lem6} (2) and using Lemma \ref{lem5} (2) on the short exact sequence (2) and (4), (5), we have
\[
\mbox{reg}\,(I_n:x_{n}^{w_n})\leq \sum\limits_{i=1}^{n-1}w_i-(n-1)+1+w+1.\eqno (6)
\]
Again by using Lemma \ref{lem5} (2) on the short exact sequence (1),
and (3), (6), we have
\[
\mbox{reg}\,(I_n)\leq\sum\limits_{i=1}^{n}w_i-(n-1)+1+w+1.
\]
\end{proof}

\medskip
Now we are ready to present the main result of this section
\begin{Theorem}\label{thm3}
Let $n\geq 2$ be an integer and $P_n$  a vertex-weighted oriented line graph, let $I(P_n)=(x_1x_2^{w_2},x_2x_3^{w_3},\ldots,x_{n-1}x_n^{w_n})$ the  edge ideal of $P_n$ with $w_i\geq 2$  for any $2\leq i\leq n-1$. Then
$$\mbox{reg}\,(I(P_n)^{t})=\sum\limits_{i=1}^{n}w_i-|E(P_n)|+1+(t-1)(w+1)\ \ \  \mbox{for all}\ \ t\geq 1,$$
where $w=\mbox{max}\,\{w_i\mid 1\leq i\leq n\}$.
\end{Theorem}
\begin{proof}
It is sufficient to show $\mbox{reg}\,(S/I(P_n)^{t})=\sum\limits_{x\in V(P_n)}w(x)-|E(P_n)|+1+(t-1)(w+1)-1$
for any $t\geq 1$ by Lemma \ref{lem6} (2).
We use  induction  on $n$ and $t$. The case $n=2$ is obvious. Assume that $n\geq 3$. For $t=1$, the statement is true by \cite[Theorem 3.5]{Z3}.
Now we consider  the case $t=2$. Let $J$ be the polarization of $I(P_n)^2$, then
\begin{eqnarray*}
\mathcal{G}(J)&=&\{x_{11}x_{12}\!\!\prod\limits_{j=1}^{2w_2}x_{2j},x_{21}x_{22}\!\!\prod\limits_{j=1}^{2w_3} x_{3j},\!\ldots,\!x_{n-2,1}x_{n-2,2}\!\!\!\!\prod\limits_{j=1}^{2w_{n-1}}\!\!\!x_{n-1,j},x_{n-1,1}x_{n-1,2}\!\!\prod\limits_{j=1}^{2w_n}x_{nj},\\
&&(x_{11}\!\!\prod\limits_{j=1}^{w_2+1}\!\!x_{2j})\prod\limits_{j=1}^{w_3}\!\!x_{3j},(x_{21}\!\!\prod\limits_{j=1}^{w_3+1}\!\!x_{3j})
\prod\limits_{j=1}^{w_4}\!\!x_{4j},
\ldots,(x_{n-2,1}\!\!\!\!\prod\limits_{j=1}^{w_{n-1}+1}\!\!\!\!x_{n-1,j})\!\!\prod\limits_{j=1}^{w_n}x_{nj},\\
&&(x_{11}\prod\limits_{j=1}^{w_2}x_{2j})(x_{31}\prod\limits_{j=1}^{w_{4}}x_{4,j}),\ldots,
(x_{n-3,1}\prod\limits_{j=1}^{w_{n-2}}x_{n-2,j})(x_{n-1,1}\prod\limits_{j=1}^{w_{n}}x_{n,j}),\ldots,\\
& &
(x_{11}\prod\limits_{j=1}^{w_2}x_{2j})(x_{n-2,1}\prod\limits_{j=1}^{w_{n-1}}x_{n-1,j}),(x_{21}\prod\limits_{j=1}^{w_3} x_{3j})(x_{n-1,1}\prod\limits_{j=1}^{w_n} x_{nj}),\\
& &(x_{11}\prod\limits_{j=1}^{w_2}x_{2j})(x_{n-1,1}\prod\limits_{j=1}^{w_n} x_{nj})\}.
 \end{eqnarray*}
 For $1\leq i\leq n-1$, we let $K_i=(x_{n-i,1}x_{n-i,2}\!\!\!\!\!\prod\limits_{j=1}^{2w_{n-i+1}}\!\!\!\!\!x_{n-i+1,j})$,
\begin{eqnarray*}
J_i&=&(x_{11}x_{12}\!\!\prod\limits_{j=1}^{2w_2}x_{2j},x_{21}x_{22}\!\!\prod\limits_{j=1}^{2w_3} x_{3j},\!\ldots,x_{n-i-1,1}x_{n-i-1,2}\!\!\!\!\prod\limits_{j=1}^{2w_{n-i}} \!\!\!\!x_{n-i,j},\\
& &\Widehat{x_{n-i,1}x_{n-i,2}\!\!\!\!\prod\limits_{j=1}^{2w_{n-i+1}} \!\!\!\!x_{n-i+1,j}},\ldots,
\!\Widehat{x_{n-2,1}x_{n-2,2}\!\!\!\!\prod\limits_{j=1}^{2w_{n-1}}\!\!\!x_{n-1,j}},
\Widehat{x_{n-1,1}x_{n-1,2}\!\!\prod\limits_{j=1}^{2w_n}x_{nj}},\\
&&(x_{11}\!\!\prod\limits_{j=1}^{w_2+1}\!\!x_{2j})\prod\limits_{j=1}^{w_3}\!\!x_{3j},(x_{21}\!\!\prod\limits_{j=1}^{w_3+1}\!\!x_{3j})
\prod\limits_{j=1}^{w_4}\!\!x_{4j},
\ldots,(x_{n-2,1}\!\!\!\!\prod\limits_{j=1}^{w_{n-1}+1}\!\!\!\!x_{n-1,j})\!\!\prod\limits_{j=1}^{w_n}x_{nj},\ldots,\\
& &(x_{11}\prod\limits_{j=1}^{w_2}x_{2j})(x_{n-2,1}\prod\limits_{j=1}^{w_{n-1}}x_{n-1,j}),(x_{21}\prod\limits_{j=1}^{w_3} x_{3j})(x_{n-1,1}\prod\limits_{j=1}^{w_n} x_{nj}),\\
& &(x_{11}\prod\limits_{j=1}^{w_2}x_{2j})(x_{n-1,1}\prod\limits_{j=1}^{w_n} x_{nj})),
 \end{eqnarray*}
 where $\Widehat{x_{n-i,1}x_{n-i,2}\!\!\!\!\prod\limits_{j=1}^{2w_{n-i+1}} \!\!\!\!x_{n-i+1,j}}$  denotes the element
 $x_{n-i,1}x_{n-i,2}\!\!\!\!\prod\limits_{j=1}^{2w_{n-i+1}} \!\!\!\!x_{n-i+1,j}$ being
 omitted from  $J_i$.
Then we have
$$J_{i-1}=J_{i}+K_{i}\ \  \text{and}\ \  K_{i}\cap J_{i}=K_{i}L_{i},,\ \  \text{for }\ \ 1\leq i\leq n-1,$$
 where $J_0=J$,
$L_1=(x_{11}\!\!\prod\limits_{j=1}^{w_2}x_{2j},x_{21}\!\!\prod\limits_{j=1}^{w_3}x_{3j},\ldots,x_{n-3,1}
\!\!\!\!\prod\limits_{j=1}^{w_{n-2}}\!\!\!\!x_{n-2,j}, x_{n-2,1}\!\!\!\!\!\!\prod\limits_{j=3}^{w_{n-1}+1}\!\!\!\!x_{n-1,j})$,\\
$L_2=(x_{11}\!\!\prod\limits_{j=1}^{w_2}\!\!x_{2j},x_{21}\!\!\prod\limits_{j=1}^{w_3}\!\!x_{3j},
\ldots,x_{n-4,1}\!\!\prod\limits_{j=1}^{w_{n-3}}\!\!x_{n-3,j}, x_{n-3,1}\!\!\!\!\!\prod\limits_{j=3}^{w_{n-2}+1}
\!\!\!\!\!x_{n-2,j},\!\!\prod\limits_{j=1}^{w_{n}}\!\!x_{nj})
$,\\
$L_i=(x_{11}\prod\limits_{j=1}^{w_2}x_{2j},x_{21}\!\!\prod\limits_{j=1}^{w_3}\!\!x_{3j},\ldots,
x_{n-i-2,1}\!\!\!\!\!\prod\limits_{j=1}^{w_{n-i-1}} \!\!\!\!\!x_{n-i-1,j},x_{n-i-1,1}\!\!\!\!\!\prod\limits_{j=3}^{w_{n-i}+1}\!\!\!\!\!x_{n-i,j},
\prod\limits_{j=1}^{w_{n-i+2}}\!\!x_{n-i+2,j},\\
x_{n-i+2,1}\!\!\prod\limits_{j=1}^{w_{n-i+3}}
\!\!\!x_{n-i+3,j},\ldots,x_{n-1,1}\!\prod\limits_{j=1}^{w_n}\!x_{nj})\ \text{for}\  3\leq i\leq n-1$
and  variables that appear in $K_{i}$ and $L_{i}$ are different  for $1\leq i\leq n-1$.

By Lemma \ref{lem2}, \cite[Theorem 3.5]{Z3} and the proof of \cite[Theorem 4.1]{Z3},  we obtain
\begin{eqnarray*}
\mbox{reg}\,(K_1\cap J_1)&=&\mbox{reg}\,(K_1L_1)=\mbox{reg}\,(K_1)+\mbox{reg}\,(L_1)\\
&=&2w_n+2+\sum\limits_{i=1}^{n-1}w_i-(n-2)+1-1\\
&=&\sum\limits_{i=1}^{n}w_i-(n-1)+1+w_n+2,\\
\mbox{reg}\,(K_2\cap J_2)&=&2w_{n-1}+2+w_n-1+\sum\limits_{i=1}^{n-2}w_i-(n-3)+1-1\\
&=&\sum\limits_{i=1}^{n}w_i-(n-1)+1+w_{n-1}+2,
\end{eqnarray*}
\begin{eqnarray*}
\mbox{reg}\,(K_i\cap J_i)\!\!\!&=&\!\!\!2w_{n-i+1}+2+\!\sum\limits_{i=1}^{n-i}w_i-(n-i-1)
+\!\!\!\sum\limits_{i={n-i+2}}^{n}\!\!\!w_i-(n-(n-i+2))-1\\
&=&\sum\limits_{i=1}^{n}w_i-(n-1)+1+w_{n-i+1}+2,\ \text{for }\  3\leq i\leq n-1.
\end{eqnarray*}
 In brief, for $1\leq i\leq n-1$, we have
\[
\mbox{reg}\,(K_i\cap J_i)=\sum\limits_{i=1}^{n}w_i-(n-1)+1+w_{n-i+1}+2.\eqno(1)
\]
Notice that all $K_i$ have linear resolutions for $1\leq i\leq n-1$, it follows that
$J_{i-1}=J_{i}+K_{i}$ is a Betti splitting and $\mbox{reg}\,(K_i)=2w_{n-i+1}+2$.
By Lemma \ref{lem1}, we obtain, for $1\leq i\leq n-1$,
\[
\mbox{reg}\,(J_{i-1})=\mbox{max}\,\{\mbox{reg}\,(J_{i}),\mbox{reg}\,(K_{i}), \mbox{reg}\,(J_{i}\cap K_{i})-1\}.\eqno(2)
\]
By repeated use of  the above equalities (2) and Lemma \ref{lem4} (2), we can obtain
\begin{eqnarray*}
\mbox{reg}\,(I(D)^2)&=&\mbox{reg}\,(J)=\mbox{max}\,\{\mbox{reg}\,(J_1),\mbox{reg}\,(K_1),\mbox{reg}\,(K_1\cap J_1)-1\}\\
&=&\mbox{max}\,\{\mbox{reg}\,(J_{n-1}),\mbox{reg}\,(K_j),\mbox{reg}\,(K_j\cap J_j)-1,1\leq j\leq n-1\}.\hspace{0.8cm} (3)
\end{eqnarray*}
 Notice that the ideal $I_n$ in Lemma  \ref{lem11} is the polarization $J_{n-1}^{\mathcal {P}}$ of $J_{n-1}$. By
 Lemmas \ref{lem4} (2) and \ref{lem11}, we have
\[
\mbox{reg}\,(J_{n-1})=\mbox{reg}\,(J_{n-1}^{\mathcal {P}})=\mbox{reg}\,(I_n)\leq \sum\limits_{i=1}^{n}w_i-(n-1)+1+(w+1).\eqno(4)
\]
Thus, by  equalities (1), (3) and (4), we obtain
\begin{eqnarray*}
\mbox{reg}\,(I(P_n)^2)&=&\mbox{max}\,\{\mbox{reg}\,(J_{n-1}),\mbox{reg}\,(K_j),\mbox{reg}\,(K_j\cap J_j)-1,1\leq j\leq n-1\}\\
&=&\mbox{max}\,\{\sum\limits_{i=1}^{n}w_i-(n-1)+1+(w+1),2w_{n-i+1}+2,\\
& &\sum\limits_{i=1}^{n}w_i-(n-1)+1+(w_{n-i+1}+1), 1\leq i \leq n-1\}\\
&=&\sum\limits_{i=1}^{n}w_i-(n-1)+1+(w+1),
\end{eqnarray*}
where the second equality holds because of $\mbox{reg}\,(J_{n-1})\leq \sum\limits_{i=1}^{n}w_i-(n-1)+1+(w+1)$
and $w=\mbox{max}\,\{w_i\mid 1\leq i\leq n\}$.

Finally  we assume that $t\geq 3$. Consider the   short exact sequences
 $$0\longrightarrow \frac{S}{(I(P_n)^{t}:x_{n}^{w_n})}(-w_n)\stackrel{ \cdot x_{n}^{w_n}}\longrightarrow \frac{S}{I(P_n)^{t}}\longrightarrow \frac{S}{(I(P_n)^{t},x_{n}^{w_n})}\longrightarrow 0\eqno(5)$$
$$0\longrightarrow \frac{S}{\!(I(P_n)^{t}\!\!:\!x_{n-1}x_{n}^{w_n})}(-1)\stackrel{\cdot x_{n-1}}\longrightarrow \frac{S}{\!(I(P_n)^{t}\!\!:\!x_{n}^{w_n}\!)}\longrightarrow \frac{S}{((I(P_n)^{t}\!\!:\!x_{n}^{w_n}),\!x_{n-1}\!)}\longrightarrow 0.\eqno(6)$$

\vspace{5mm}
\hspace{-4mm}Notice that $(I(P_n)^{t},x_{n}^{w_n})=(I(P_{n-1})^{t},x_{n}^{w_n})$, $(I(P_n)^{t}:x_{n-1}x_n^{w_n})=I(P_n)^{t-1}$ and
$((I(P_n)^{t}:x_n^{w_n}),x_{n-1})=(I(P_{n-2})^{t},x_{n-1})$ by Lemmas \ref{lem8} $\sim$ \ref{lem10}. Thus, by induction hypotheses on $n$ and  $t$, and Lemma \ref{lem2} (1) and Lemma \ref{lem6} (2), we get
\begin{eqnarray*}
\mbox{reg}\,(\frac{S}{(I(P_n)^t,x_{n}^{w_n})})\!\!\!&=&\!\!\!\mbox{reg}\,(\frac{S}{(I(P_{n-1})^{t},x_{n}^{w_n})})
=\mbox{reg}\,(S'/(I(P_{n-1})^{t}))+\mbox{reg}\,(k[x_n]/(x_{n}^{w_n}))\\
&=&\!\!\!\!\sum\limits_{i=1}^{n-1}w_i-|E(P_{n-1})|+1+(t-1)(w'+1)-1+w_n-1\\
&=&\!\!\!\!\sum\limits_{i=1}^{n}w_i-|E(P_n)|+1+(t-1)(w'+1)-1, \hspace{2.9cm} (7)
\end{eqnarray*}
where $w'=\mbox{max}\,\{w_i\mid 1\leq i\leq n-1\}$ and $S'=k[x_1,\ldots,x_{n-1}]$;
\[
\mbox{reg}\,(\frac{S}{(I(P_n)^{t}:x_{n-1}x_n^{w_n})}\!)=\mbox{reg}\,(\frac{S}{I(P_n)^{t-1}})\\
=\sum\limits_{i=1}^{n}w_i-|E(P_n)|+1+(t-2)(w+1)-1,\eqno (8)
\]
and
\begin{eqnarray*}
\mbox{reg}\,(\frac{S}{((I(P_n)^{t}:x_n^{w_n}),x_{n-1})})&=&\mbox{reg}\,(\frac{S}{(I(P_{n-2})^{t},x_{n-1})})
=\mbox{reg}\,(\frac{S''}{I(P_{n-2})^{t}})\\
&=&\!\!\sum\limits_{i=1}^{n-2}w_i-|E(P_{n-2})|+1+(t-1)(w''+1)-1 \hspace{0.7cm} (9)
\end{eqnarray*}
where the second equality holds by Lemma \ref{lem7} (2),  here $w''=\mbox{max}\,\{w_i\mid 1\leq i\leq n-2\}$ and $S''=k[x_1,x_2,\ldots,x_{n-2}]$.
Let
\[
\alpha=\mbox{reg}\,(\frac{S}{(I(P_n)^t,x_{n}^{w_n})}),\ \ \ \beta=\mbox{reg}\,(\frac{S}{(I(P_n)^{t}:x_{n-1}x_n^{w_n})}(-w_n-1))
\]
 and
\[
\gamma=\mbox{reg}\,(\frac{S}{((I(P_n)^{t}:x_n^{w_n}),x_{n-1})}(-w_n)),
\]
 then, by equality (8) and (9), we get
\[\beta-\gamma=(t-1)(w-w'')+w_n-w-2+w_{n-1}.
\]

If $w=w'$, then  $\alpha\geq \mbox{max}\,\{\beta,\gamma\}$ by comparing (7), (8) and (9). Using Lemma \ref{lem5} (3) on the short exact sequence (6), we obtain
\begin{eqnarray*}
\mbox{reg}\,(\frac{S}{(I(P_n)^t,x_{n}^{w_n})})&\geq&\mbox{reg}\,(\frac{S}{(I(D)^{t}:x_{n}^{w_n})}(-w_n)).
\end{eqnarray*}
Again using Lemma \ref{lem5} (3) on the short exact sequence (5), we get
$$\mbox{reg}\,(S/I(P_n)^{t})=\mbox{reg}\,(\frac{S}{(I(P_n)^{t},x_{n}^{w_n})})
=\sum\limits_{i=1}^{n}w_i-|E(P_n)|+1+(t-1)(w+1)-1.$$
If $w>w'$, then $w=w_n$ and $w>w''$. This implies $\beta-\gamma>1$  because of $t\geq 3$.
Using Lemma \ref{lem5} (4) on the  exact sequence (6), we obtain
\begin{eqnarray*}
\mbox{reg}\,(\frac{S}{(I(P_n)^{t}:x_{n}^{w_n})})&=&\mbox{reg}\,(\frac{S}{(I(P_n)^{t}:x_{n-1}x_n^{w_n})}(-1))\\
&=&\sum\limits_{i=1}^{n}w_i-|E(P_n)|+1+(t-2)(w+1).\hspace{2.8cm} (10)
\end{eqnarray*}
By comparing equalities (7) and (10), we get
\begin{eqnarray*}
\mbox{reg}\,(\frac{S}{(I(P_n)^{t}:x_{n}^{w_n})}(-w_n))\!&>\!&\mbox{reg}\,(\frac{S}{(I(P_n)^{t},x_{n}^{w_n})})+1
 \end{eqnarray*}
because of $t\geq 3$.
Therefore, by using Lemma \ref{lem5} (4) on the exact sequence (5), and equality (10), we have
\begin{eqnarray*}
\mbox{reg}\,(S/I(P_n)^{t})&=&\mbox{reg}\,(\frac{S}{(I(P_n)^{t}:x_{n}^{w_n})}(-w_n))\\
&=&\sum\limits_{i=1}^{n}w_i-|E(P_n)|+1+(t-1)(w+1)-1.
\end{eqnarray*}
The proof is completed.
\end{proof}

\medskip
The following example shows that the assumption in Theorem \ref{thm3} that
$w_i\geq 2$  for any $2\leq i\leq n-1$ cannot be dropped.
\begin{Example}  \label{example6}
Let $I(D)=(x_1x_2^{5},x_2x_3,x_3x_4^{8})$ be the edge ideal of vertex-weighted oriented line graph $D=(V(D),E(D),w)$  with $w_1=w_3=1$,
$w_2=5$ and  $w_4=8$. By using CoCoA, we get $\mbox{reg}\,(I(D)^2)=18$.  But  we have
$\mbox{reg}\,(I(D)^{2})=\sum\limits_{i=1}^{4}w_i-|E(P_n)|+1+(w_4+1)=22$
 by Theorem \ref{thm3}.
\end{Example}

The following example shows that  the regularity of the powers of
the edge ideals of vertex-weighted oriented line graphs are related to
direction selection in Theorem \ref{thm3}.
\begin{Example}  \label{example6}
Let $I(D)=(x_1x_2^{5},x_3x_2^{5},x_3x_4^{8},x_5x_4^8,x_5x_6^{2})$ be the edge ideal of a vertex-weighted oriented line graph $D=(V(D),E(D),w)$  with $w_1=w_3=w_5=1$, $w_2=5$, $w_4=8$ and $w_6=2$. By using CoCoA, we obtain  $\mbox{reg}\,(I(D)^2)=17$.
But  we have
$\mbox{reg}\,(I(D)^{2})=\sum\limits_{i=1}^{6}w_i-|E(P_n)|+1+(w_4+1)=23$
 by Theorem \ref{thm3}.
\end{Example}

\section{projective dimension and  regularity of  powers of edge ideals of vertex-weighted rooted forests}

\hspace{3mm}In this section, we will give some formulas for the projective dimension and the regularity of the
powers of the edge ideals of vertex-weighted rooted forests. We shall start from oriented star graph.

\begin{Theorem}\label{thm4}
Let $D=(V(D), E(D),w)$ be a weighted oriented star graph. If its edge set $E(D)$ is one of the following three cases $\{x_1x_2,x_1x_3,\ldots,x_1x_n\}$, $\{x_2x_1,x_3x_1,\\
\ldots,x_nx_1\}$, and $\{x_1x_2,x_2x_3,\ldots,x_2x_n\}$ where $w_2\geq 2$. Then
$$\mbox{reg}\,(I(D)^{t})=\sum\limits_{i=1}^{n}w_i-|E(D)|+1+(t-1)(w+1)\hspace{1cm}  \mbox{for all}\ \ t\geq 1,$$
where $w=\mbox{max}\,\{w_i\mid 1\leq i\leq n\}$.
\end{Theorem}
\begin{proof}
The  cases $E(D)=\{x_1x_2,x_1x_3,\ldots,x_1x_n\}$ and $E(D)=\{x_2x_1,x_3x_1,\ldots,x_nx_1\}$ can be shown by similar arguments,  we only consider the case $E(D)=\{x_1x_2,x_1x_3,\\
\ldots,x_1x_n\}$, then
$$I(D)^t=(x_1^t)(x_2^{w_2},x_3^{w_3},\ldots,x_n^{w_n})^t.$$
From  Lemmas \ref{lem3} (2), \ref{lem6} (2), \ref{lem2} (1) and Theorem  \ref{thm1}, it follows that
$$\mbox{reg}\,(I(D)^{t})=t+\sum\limits_{i=2}^{n}w_i-(n-2)+(t-1)w=\sum\limits_{i=1}^{n}w_i-|E(D)|+1+(t-1)(w+1).$$

If $E(D)=\{x_1x_2,x_2x_3,\ldots,x_2x_n\}$, then $$I(D)^t=(x_1x_2^{w_2},x_2x_3^{w_3},\ldots,x_2x_n^{w_n})^t.$$

It is  sufficient to show  $\mbox{reg}\,(\frac{S}{I(D)^{t}})=\sum\limits_{i=1}^{n}w_i-|E(D)|+1+(t-1)(w+1)-1$ for any $t\geq 1$ by Lemma \ref{lem6} (2).
We prove this statement by induction on $n$ and $t$. The cases  $n=2$ and $n=3$ follow from Theorem \ref{thm3}.
Assume that $n\geq 4$. The case $t=1$ follows from \cite[Theorem 3.1]{Z3}.
Now assume that $t\geq 2$. Without lose of generality, we may assume
 $w_{n}\leq w_{n-1}$. Consider  the following short exact sequences
 $$0\longrightarrow \frac{S}{(I(D)^{t}:x_{n}^{w_n})}(-w_n)\stackrel{ \cdot x_{n}^{w_n}} \longrightarrow \frac{S}{I(D)^{t}}\longrightarrow \frac{S}{(I(D)^{t},x_{n}^{w_n})}\longrightarrow 0 \eqno (1)$$
and
$$0\longrightarrow \frac{S}{((I(D)^{t}:x_{n}^{w_n}):x_2)}(-1)\stackrel{\cdot x_2} \longrightarrow \frac{S}{(I(D)^{t}:x_{n}^{w_n})}\longrightarrow \frac{S}{((I(D)^{t}:x_{n}^{w_n}),x_2)} \longrightarrow 0.\eqno (2)$$
Note that $(I(D)^{t},x_{n}^{w_n})=(I(D\setminus x_{n})^{t},x_{n}^{w_n})$ by Lemma \ref{lem8}.
Thus, by  Lemma \ref{lem2} (1),  Lemma \ref{lem6} (2) and induction hypothesis on $n$, we have
\begin{eqnarray*}
\mbox{reg}\,(\frac{S}{(I(D)^{t},x_{n}^{w_n})})&=&\mbox{reg}\,(\frac{S}{I(D\setminus x_{n})^{t},x_{n}^{w_n}})
=\mbox{reg}\,(\frac{S'}{I(D\setminus x_{n})^{t}})+\mbox{reg}\,(k[x_n]/(x_{n}^{w_n}))\\
&=&\sum\limits_{i=1}^{n-1}w_i-|E(D\setminus x_n)|+1+(t-1)(w'+1)-1+w_n-1\\
&=&\sum\limits_{i=1}^{n}w_i-|E(D)|+1+(t-1)(w+1)-1.\hspace{2.7cm} (3)
\end{eqnarray*}
where the last equality holds because of $w':=\mbox{max}\{w_i\mid 1\leq i\leq n-1\}=w$ and $S'=k[x_1,x_2,\ldots,x_{n-1}]$.

Since $(I(D)^{t}:x_2x_{n}^{w_n})=I(D)^{t-1}$ by Lemma \ref{lem9}, by using induction hypothesis on $t$ and Lemma  \ref{lem6} (2),  we get
\[
\mbox{reg}\,(\frac{S}{I(D)^{t}:x_2x_{n}^{w_n}})=\mbox{reg}\,(\frac{S}{I(D)^{t-1}})
=\sum\limits_{i=1}^{n}w_i-|E(D)|+1+(t-2)(w+1)-1.\eqno (4)
\]
Notice that $((I(D)^{t}:x_{n}^{w_n}),x_2)=(I(D\setminus x_2)^{t},x_2)=(x_2)$ by Lemma \ref{lem10}, it follows that
\[
\mbox{reg}\,(\frac{S}{((I(D)^{t}:x_{n}^{w_n}),x_2)})=\mbox{reg}\,(S/(x_2))=0.\eqno (5)
\]
By comparing the equalities (4) and (5), we get
\begin{eqnarray*}
\mbox{reg}\,(\frac{S}{(I(D)^{t}:x_2x_{n}^{w_n})}(-1))&>&\mbox{reg}\,(\frac{S}{((I(D)^{t}:x_{n}^{w_n}),x_2)})+1.
\end{eqnarray*}
By the above inequality, the exact sequence (2) and Lemma  \ref{lem5} (4), we have
\begin{eqnarray*}
\mbox{reg}\,(\frac{S}{(I(D)^{t}:x_{n}^{w_n})})&=&\mbox{reg}\,(\frac{S}{((I(D)^{t}:x_2x_{n}^{w_n})}(-1))\\
&=&\sum\limits_{i=1}^{n}w_i-|E(D)|+1+(t-2)(w+1),\hspace{3.0cm} (6)
\end{eqnarray*}
where the last equality holds from (4).
By comparing with the equalities (3) and (6), we have
\begin{eqnarray*}
\mbox{reg}\,(\frac{S}{(I(D)^{t},x_{n}^{w_n})})&\geq&\mbox{reg}\,(\frac{S}{(I(D)^{t}:x_{n}^{w_n})}(-w_n)).
 \end{eqnarray*}
Therefore, by the  exact sequence (1), Lemma  \ref{lem5} (3) and Lemma \ref{lem6} (2), we obtain that
$$\mbox{reg}\,(S/I(D)^{t})=\mbox{reg}\,(\frac{S}{(I(D)^{t},x_{n}^{w_n})})=\sum\limits_{i=1}^{n}w_i-|E(D)|+1+(t-1)(w+1)-1,$$
where the last equality holds from (3).
The proof is completed.
\end{proof}

\medskip
Now, we prove the main results of this section.
\begin{Theorem}\label{thm5}

\end{Theorem}
\begin{proof}
We prove this statement by induction on $|V(D)|$ and $t$. The case $|V(D)|=2$ follows from Theorem \ref{thm3}. Assume that $|V(D)|\geq 3$. The case $t=1$ follows from \cite[Theorem 3.5]{Z3}.
Now assume that $t\geq 2$. If
 $D$ is an oriented line graph, it follows from  Theorem \ref{thm3}. Otherwise, there are at least two leaves in $D$. We may
suppose both $x$ and $z$ are leaves of $D$ such that  $w_z\leq w_x$ and $N_{D}^{-}(z)=\{y\}$.  If $D$ is an oriented star graph, then this  statement holds by Theorem \ref{thm4}.

If there exists a connected component $D_1$ of $D$, which is a line graph with only one edge $\{y,z\}$.
Then $I(D)=(I(D\setminus z), yz^{w_z})$  and
$$I(D)^{t}=(I(D\setminus z),yz^{w_z})^{t}=I(D\setminus z)^{t}+yz^{w_z}I(D)^{t-1}.$$
Thus there exists a surjection:
$$\phi :\ I(D)^{t-1}(-w_{z}-1)\oplus I(D\setminus z)^{t} \overset{\cdot (yz^{w_z},1)}\longrightarrow  I(D)^{t}.$$
Since $yz^{w_z}$ is a regular  element of $S/I(D\setminus z)$, the kernel of $\phi$ is  $yz^{w_z}I(D\setminus z)^{t}$.
Thus we  have the  short exact sequence
$$0\longrightarrow I(D\setminus z)^{t}(-w_{z}-1)\longrightarrow I(D)^{t-1}(-w_{z}-1)\oplus I(D\setminus z)^{t}\stackrel{\cdot (yz^{w_z},1)}\longrightarrow I(D)^t \longrightarrow 0.$$
By induction hypotheses on $t$ and $|V(D)|$, we obtain that
\begin{eqnarray*}\mbox{reg}\,(I(D)^{t-1}(-w_{z}-1))&=&\mbox{reg}\,(I(D)^{t-1})+w_{z}+1\\
&=&\sum\limits_{x\in V(D)}w(x)-|E(D)|+1+(t-2)(w+1)+w_{z}+1\\
&=&\sum\limits_{x\in V(D)}w(x)-|E(D)|+1+(t-1)(w+1)+w_{z}-w\\
\end{eqnarray*}
and
\begin{eqnarray*}\mbox{reg}\,(I(D\setminus z)^{t})&=&\sum\limits_{x\in V(D\setminus z)}w(x)-|E(D\setminus z)|+1+(t-1)(w'+1)\\
&=&\sum\limits_{x\in V(D)}w(x)-|E(D)|+1+(t-1)(w+1)+1-w_{z}-1
\end{eqnarray*}
 where the second equality holds in formula for $\mbox{reg}\,(I(D\setminus z)^{t})$  because of $w_z\leq w_x$ and $w'=\mbox{max}\,\{w(x)\mid x\in V(D\setminus z)\}$,
\begin{eqnarray*}\mbox{reg}\,(I(D\setminus z)^{t}(-w_{z}-1))&=&\mbox{reg}\,(I(D\setminus z)^{t})+w_{z}+1\\
&=&\!\!\!\!\!\!\sum\limits_{x\in V(D)}\!\!\!w(x)-|E(D)|+1+(t-1)(w+1)-w_{z}+w_{z}+1\\
&=&\!\!\!\!\!\sum\limits_{x\in V(D)}\!\!\!w(x)-|E(D)|+1+(t-1)(w+1)+1.
\end{eqnarray*}
It follows that \[
\mbox{reg}\,(I(D\setminus z)^{t}(-w_{z}-1))>\mbox{max}\,\{\mbox{reg}\,(I(D)^{t-1}(-w_{z}-1)),\mbox{reg}\,(I(D\setminus z)^{t})\}.
\]
Thus the result follows from Lemma \ref{lem5} (5).

Otherwise, we will show $\mbox{reg}\,(\frac{S}{I(D)^{t}})=\!\!\!\sum\limits_{x\in V(D)}\!\!\!w(x)-|E(D)|+1+(t-1)(w+1)-1$ for any $t\geq 1$.
Thus the conclusion follows  from Lemma \ref{lem6} (2).

We consider the following short exact sequences
 $$0\longrightarrow \frac{S}{(I(D)^{t}:z^{w_z})}(-w_z)\stackrel{ \cdot z^{w_z}} \longrightarrow \frac{S}{I(D)^{t}}\longrightarrow \frac{S}{(I(D)^{t},z^{w_z})}\longrightarrow 0    \eqno(1)$$
and
$$0\longrightarrow \frac{S}{((I(D)^{t}:z^{w_z}):y)}(-1)\stackrel{\cdot y} \longrightarrow \frac{S}{(I(D)^{t}:z^{w_z})}\longrightarrow \frac{S}{((I(D)^{t}:z^{w_z}),y)} \longrightarrow 0 \eqno(2)$$
Notice that $(I(D)^{t},z^{w_z})=(I(D\setminus z)^{t},z^{w_z})$, $(I(D)^{t}:yz^{w_z})=I(D)^{t-1}$
and $((I(D)^{t}:z^{w_z}),y)=(I(D\setminus y)^t,y)$ by Lemmas \ref{lem7} $\sim$ \ref{lem9}. Thus   by Lemma \ref{lem2} (1) and  induction hypotheses on $|V(D)|$ and $t$, we obtain that
\begin{eqnarray*}
\mbox{reg}\,(\frac{S}{(I(D)^t,z^{w_z})})&=&\!\!\mbox{reg}\,(\frac{S}{(I(D\setminus z)^{t},z^{w_z})})
=\!\mbox{reg}\,(S'/(I(D\setminus z)^{t}))+\mbox{reg}\,(k[z]/(z^{w_z}))\\
&=&\!\!\sum\limits_{x\in V(D\setminus z)}\!\!w(x)-|E(D\setminus z)|+1+(t-1)(w'+1)-1+w_z-1\\
&=&\sum\limits_{x\in V(D)}w(x)-|E(D)|+1+(t-1)(w+1)-1 \hspace{1.7cm}\  \ (3)
\end{eqnarray*}
where $w'=\mbox{max}\,\{w(x)\mid x\in V(D\setminus z)\}$ and $S'$ is a polynomial ring over a field with the variable set $V(D\setminus z)$,
$$
\mbox{reg}\,(\frac{S}{(I(D)^{t}:yz^{w_z})})=\mbox{reg}\,(\frac{S}{(I(D)^{t-1})})=\!\!\!\!\!
\sum\limits_{x\in V(D)}\!\!\!\!\!w(x)-|E(D)|+1+(t-2)(w+1)-1.\eqno(4)
$$

Let $N_D(y)=N_D^{-}(y)\cup N_D^{+}(y)$, where $N_D^{-}(y)=\{x_1\}$, $N_D^{+}(y)=\{z,x_2,\ldots,x_{\ell},x_{{\ell}+1},\\
\ldots,x_m\}$ such that $\{z,x_2,\ldots,x_{\ell}\}$ being leaves of $D$ and
$\{x_{\ell+1},\ldots,x_m\}$ not being leaves of $D$ or $N_D(y)=N_D^{+}(y)=\{z,x_1,x_2,\ldots,x_{\ell},x_{{\ell}+1},\ldots,x_m\}$ such that $\{z,x_1,\ldots,x_{\ell}\}$ being leaves of $D$ and $\{x_{\ell+1},
\ldots,x_m\}$ not being leaves of $D$. The second case means that $w_y=1$. Thus by Lemma \ref{lem2} (1) and induction hypothesis on $|V(D)|$, we obtain

(1)  If $N_D^{-}(y)=\{x_1\}$, then
\begin{eqnarray*}
\mbox{reg}\,(\frac{S}{((I(D)^{t}:z^{w_z}),y)})
&=&\sum\limits_{x\in V(D\setminus y)}w(x)-|E(D\setminus y)|+1+(t-1)(w''+1)-1\\
&=&\sum\limits_{x\in V(D)}w(x)-|E(D)|+1+(t-1)(w''+1)+(m+1)\\
&-&((w_{\ell+1}-1)+(w_{\ell+2}-1)+\cdots+(w_m-1))\\
&-&(w_y+w_z+w_1+w_2+\cdots+w_{\ell})-1,\hspace{2.5cm}\  \ (5)
\end{eqnarray*}
(2) If $N_D^{-}(y)=\emptyset$, then
\begin{eqnarray*}
\mbox{reg}\,(\frac{S}{((I(D)^{t}:z^{w_z}),y)})&=&\mbox{reg}\,(S/((I(D\setminus y)^t,y))\\
&=&\mbox{reg}\,(S''/(I(D\setminus y)^t))+\mbox{reg}\,(k[y]/(y))\\
&=&\sum\limits_{x\in V(D\setminus y)}w(x)-|E(D\setminus y)|+1+(t-1)(w''+1)-1\\
&=&\sum\limits_{x\in V(D)}w(x)-|E(D)|+1+(t-1)(w''+1)+(m+1)\\
&-&((w_{\ell+1}-1)+(w_{\ell+1}-1)+\cdots+(w_m-1))\\
&-&(w_y+w_z+w_2+\cdots+w_{\ell})-1,\hspace{3.5cm}\  \ (6)
\end{eqnarray*}
 where the last equality holds because $\{x_{\ell+1},\ldots,x_m\}$ being roots of $D\setminus y$, here $w''=\mbox{max}\,\{w(x)\mid x\in V(D\setminus y)\}$ and $S''$ is a polynomial ring over a field with the variable set $V(D\setminus y)$.

 Using Lemma \ref{lem5} (2) on the short exact sequence (2), we have
 \[
 \mbox{reg}\,(\frac{S}{(I(D)^{t}:z^{w_z})})\leq \mbox{max}\{\mbox{reg}\,(\frac{S}{((I(D)^{t}:z^{w_z}):y)}(-1)),\mbox{reg}\,(\frac{S}{((I(D)^{t}:z^{w_z}),y)})\}.
 \]

 By comparing the equalities (3), (4) (5) and (6), we get
\[
\mbox{reg}\,(\frac{S}{(I(D)^t,z^{w_z})})\geq \mbox{reg}\,(\frac{S}{(I(D)^{t}\!:yz^{w_z})}(-w_z-1)).
\]

Again using Lemma \ref{lem5} (3) on
the short exact sequence (1),  we obtain that
\[
\mbox{reg}\,(\frac{S}{I(D)^{t}})=
\mbox{reg}\,(\frac{S}{(I(D)^t,z^{w_z})})=\!\!\!\!\!\sum\limits_{x\in V(D)}\!\!\!\!\!w(x)-|E(D)|+1+(t-1)(w+1)-1.
\]
The proof is completed.
\end{proof}

\medskip
 As a consequence of Theorem \ref{thm5}, we have
\begin{Corollary}\label{cor2}
Let $D=(V(D),E(D),w)$ be a vertex-weighted rooted forest as in Theorem \ref{thm5}. Then
$$\mbox{reg}\,(I(D)^{t})=\mbox{reg}\,(I(D))+(t-1)(w+1)\hspace{1cm}  \mbox{for all}\ \ t\geq 1,$$
where $w=\mbox{max}\,\{w(x)\mid x \in V(D)\}$.
\end{Corollary}
\begin{proof}This is a direct consequence of the above theorem and \cite[Theorem 3.5]{Z3}.
\end{proof}

\medskip
\begin{Theorem}\label{thm6}
Let $D=(V(D),E(D),w)$ be a vertex-weighted rooted forest such that   $w(x)\geq 2$ if $d(x)\neq 1$. Then
$$\mbox{pd}\,(I(D)^{t})=|E(D)|-1 \ \  \mbox{for all}\ \ t\geq 1.$$
\end{Theorem}
\begin{proof}
 We prove this statement by induction  on $|E(D)|$ and $t$. The case $|E(D)|=1$ is clear. Assume that $|E(D)|\geq 2$. The case $t=1$ follows from
 \cite[Theorem 3.3]{Z3}.
Now assume that $t\geq 2$. We may suppose $z$ is a leaf of $D$ and $N_D^{-}(z)=\{y\}$.
we will show $\mbox{pd}\,(\frac{S}{I(D)^{t}})=|E(D)|$ for any $t\geq 1$.
Thus the conclusion follows  from Lemma \ref{lem6} (1).

Consider the following short exact sequences
 $$0\longrightarrow \frac{S}{(I(D)^{t}:z^{w_z})}(-w_{z})\stackrel{ \cdot z^{w_z}} \longrightarrow \frac{S}{I(D)^{t}}\longrightarrow \frac{S}{(I(D)^{t},z^{w_z})}\longrightarrow 0    \eqno(1)$$
and
$$0\longrightarrow \frac{S}{((I(D)^{t}:z^{w_z}):y)}(-1)\stackrel{\cdot y} \longrightarrow \frac{S}{(I(D)^{t}:z^{w_z})}\longrightarrow \frac{S}{((I(D)^{t}:z^{w_z}),y)} \longrightarrow 0 \eqno(2)$$
Notice that $(I(D)^{t},z^{w_z})=(I(D\setminus z)^{t},z^{w_z})$, $(I(D)^{t}:yz^{w_z})=I(D)^{t-1}$
and $((I(D)^{t}:z^{w_z}),y)=(I(D\setminus y)^t,y)$ by Lemmas \ref{lem8} $\sim$ \ref{lem10}. We set $N_D(y)=\{z,x_1,x_2,\ldots,x_m\}$, this implies $|E(D\setminus y)|=|E(D)|-(m+1)$. Thus by Lemma \ref{lem2} (2) and induction hypotheses on $|E(D)|$ and $t$, we obtain
\begin{eqnarray*}
\mbox{pd}\,(\frac{S}{(I(D)^{t},z^{w_z})})&=&\mbox{pd}\,(\frac{S}{(I(D\setminus z)^{t},z^{w_z})})
=\mbox{pd}\,(\frac{S'}{(I(D\setminus z)^{t})})+\mbox{pd}\,(\frac{k[z]}{(z^{w_z})})\\
&=&|E(D\setminus z)|+1=|E(D)|  \hspace{6.0cm} (3)
\end{eqnarray*}
where $S'$ is the polynomial ring with  variable set $V(D\setminus z)$,
\[
\mbox{pd}\,(\frac{S}{(I(D)^{t}:yz^{w_z})})=\mbox{pd}\,(\frac{S}{I(D)^{t-1}})=|E(D)| \eqno (4)
\]
and
\begin{eqnarray*}
\mbox{pd}\,(\frac{S}{(I(D)^{t}:z^{w_z}),y)})&=&\mbox{pd}\,(\frac{S}{(I(D\setminus y)^t,y)})=\mbox{pd}\,(\frac{S''}{I(D\setminus y)^t})+1\\
&=&|E(D\setminus y)|+1=|E(D)|-m.\hspace{4.2cm} (5)
\end{eqnarray*}
where $S''$ is the polynomial ring with  variable set $V(D\setminus y)$.

Using Lemma \ref{lem5} (1) on
the short exact sequence (1), (2) and equalities (3), (4) and  (5), we get
$$\mbox{pd}\,(\frac{S}{I(D)^{t}})=\mbox{pd}\,(\frac{S}{(I(D)^{t},z^{w_z})})=|E(D)|.$$
\end{proof}

An immediate consequence of the above theorem is the following corollary.
\begin{Corollary} \label{cor4}
Let $D=(V(D),E(D),w)$ be a vertex-weighted rooted forest as in Theorem \ref{thm6}. Then  $\mbox{depth}\,(I(D))=n-|E(D)|+1$.
\end{Corollary}
\begin{proof} By Auslander-Buchsbaum formula (see Theorem 1.3.3 of \cite{BH}), it follows that
$$\mbox{depth}\,(I(D))=n-\mbox{pd}\,(I(D))=n-|E(D)|+1.$$
\end{proof}

\medskip
The following  example shows that the assumption in Theorems \ref{thm5} and \ref{thm6}
that $D$ is a vertex-weighted rooted forest such that $w(x)\geq 2$ if $d(x)\neq 1$  cannot be dropped.

\begin{Example}  \label{example6}
Let $I(D)=(x_1x_2^{2},x_2x_3,x_3x_4^{2},x_5x_6^{2},x_6x_7,x_7x_8^{2})$ be the edge ideal of weighted oriented  forest $D=(V(D),E(D),w)$  with $w_1=w_3=w_5=w_7=1$ and  $w_2=w_4=w_6=w_8=2$. By using CoCoA, we obtain $\mbox{pd}\,(I(D)^2)=4$ and $\mbox{reg}\,(I(D)^2)=8$. But  we have
$\mbox{pd}\,(I(D)^2)=|E(D)|-1=5$ by Theorem \ref{thm6}
and $\mbox{reg}\,(I(D)^{2})=\sum\limits_{i=1}^{8}w_i-|E(D)|+1+(w_2+1)=10$
 by Theorem \ref{thm5}.
 \end{Example}

\medskip
 The following example shows that  the projective dimension of  the powers of
the edge ideals of vertex-weighted oriented forest are related to
direction selection in Theorem \ref{thm6}.
\begin{Example}  \label{example6}
Let $I(D)=(x_1x_2^{5},x_3x_2^{5},x_3x_4^{8},x_5x_4^8,x_5x_6^{2})$ be the edge ideal of  vertex-weighted oriented line graph $D=(V(D),E(D),w)$  with $w_1=w_3=w_5=1$, $w_2=5$, $w_4=8$ and $w_6=2$. By using CoCoA, we obtain  $\mbox{pd}\,(I(D)^2)=3$. But  we get
$\mbox{pd}\,(I(D)^2)=|E(D)|-1=4$ by Theorem \ref{thm6}.
\end{Example}

\medskip

\hspace{-6mm} {\bf Acknowledgments}

 \vspace{3mm}
\hspace{-6mm}  This research is supported by the National Natural Science Foundation of China (No.11271275) and  by foundation of the Priority Academic Program Development of Jiangsu Higher Education Institutions.


\begin{thebibliography}{99}



\bibitem{AB} A. Alilooee and A. Banerjee, Powers of edge ideals of regularity three bipartite graphs, {\it J. Commut. Algebra}, 9 (4) (2017), 441-454.


\bibitem{ABS} A. Alilooee,  A. Banerjee and S. Selvaraja, Regularity of powers of edge ideal of unicyclic graphs,
{\it arXiv: 1702.001916V4}.

\bibitem{AF1} A. Alilooee and S. Faridi, On the resolution of path ideals of cycles, {\it Comm. Algebra}, 43 (2015), 5413-5433.

\bibitem{AF2} A. Alilooee and S. Faridi, Graded Betti numbers of path ideals of cycles and lines, {\it J. Algebra Appl.}, 17 (2017), 1850011-1-17.

\bibitem{B1}  A. Banerjee, The regularity of powers of edge ideals, {\it J. Algebraic Combin.}, 41 (2014), 303-321.

\bibitem{BBH1}  A. Banerjee, S. Beyarslan, and H. T. H\`a, Regularity of edge ideals and their powers, {\it arXiv:1712.00887V2.}

\bibitem{BC} T. Biyiko\v{g}lu and Y. Civan, Projective dimension of (hyper)graphs and the Castelnuovo-Mumford regularity of bipartite graphs, {\it  arXiv:1605.02956V1}.

\bibitem{B2}  M. Brodmann,  The asymptotic nature of the analytic spread, {\it  Math. Proc. Cambridge Philos Soc.}, 86 (1979), 35-39.



\bibitem{BBH2}  A. Banerjee, S. Beyarslan, and H. T.  H\`a, Regularity of powers of edge ideals: from local properties to global bounds, {\it  arXiv:1805.01434V2}.

\bibitem{BHT} S. Beyarslan, H. T. H\`a and T. N. Trung, Regularity of powers of forests and cycles, {\it J.
Algebraic Combin.}, 42 (2015),  1077-1095.



\bibitem{BHO} R. R. Bouchat, H. T. H\`a, and A. O{'}Keefe. Path ideals of rooted trees and their graded Betti
numbers, {\it J. Comb. Theory}, Ser. A, 118 (8) (2011), 2411-2425.


\bibitem{BM} J. A. Bondy and U. S. R. Murty, {\it Graph Theory}, Springer.com, 2008.

\bibitem{BH} W. Bruns and J. Herzog, {\it Cohen-Macaulay rings}, Revised Edition, Cambridge University Press,
1998.

\bibitem{Ch} M. Chardin, Some results and questions on Castelnuovo-Mumford regularity,
in Syzygies and Hilbert functions, {\it Lect. Notes Pure Appl. Math.}, 254, 1-40.


\bibitem{Co} CoCoATeam, CoCoA: a system for doing Computations in Commutative Algebra, Avaible at
http://cocoa.dima.unige.it


\bibitem{CHT} S. D. Cutkosky, J. Herzog and N. V. Trung, Asymptotic behaviour of the Castelnuovo-Mumford regularity, {\it Compos. Math.} 118 (3) (1999), 243-261.

\bibitem{DHS} H. Dao, C. Huneke and J. Schweig, Bounds on the regularity and projective dimension of
ideals associated to graphs,  {\it J. Algebraic Combin.}, 38(1) (2013), 37-55.

\bibitem{DS1} H. Dao and J. Schweig, Projective dimension, graph domination parameters, and
independence complex homology,  {\it J. Combin. Theory}, Series A, 432 (2) (2013), 453-469.

\bibitem{DS2} H. Dao and J. Schweig, Bounding the projective dimension of a squarefree monomial
ideal via domination in clutters, {\it Proc. Amer. Math. Soc.}, 143  (2015), 555-565.

\bibitem{DS3} H. Dao and J. Schweig, Further applications of clutter domination parameters to
projective dimension,  {\it J. Algebra}, 432, (2015), 1-11.

\bibitem{E} B  Enghta, Bounds on Projective Dimension,  PhD dissertation, University of
Kansas, 2005.

\bibitem{EF} N. Erey and  S. Faridi, Betti numbers of monomial ideals via facet covers, {\it J. Pure Appl. Algebra}, 220 (2016), 1990-2000.



\bibitem{EK} S. Eliahou and M. Kervaire, Minimal resolutions of some monomial ideals, {\it J. Algebra}, 129 (1990), 1-25.

\bibitem{F1} Sara Faridi, The projective dimension of sequentially Cohen-Macaulay monomial ideals, {\it arXiv:1310.5598V2}.

\bibitem{F2} G. Fatabbi, On the resolution of ideals of fat points, {\it J. Algebra}, 242 (2001), 92-108.

\bibitem{FM} L. Fouli and  S. Morey,  A lower bound for depths of powers of edge ideals, {\it J. Algebr Comb}, 42 (2015), 829-848.


\bibitem{FHT} C. A. Francisco, H. T. H\`a and A. Van Tuyl, Splittings of monomial ideals, {\it Proc. Amer. Math. Soc.}, 137 (10) (2009),
3271-3282.

\bibitem{GBSVV} P. Gimenez, J. M. Bernal, A. Simis, R. H. Villarreal, and
C. E. Vivares, Monomial ideals and Cohen-Macaulay vertex-weighted digraphs, arXiv: 1706.00126v3.


\bibitem{HT1} H. T. H\`a and A. Van Tuyl, Splittable ideals and the resolutions of monomial ideals, {\it J. Algebra}, 309 (1) (2007), 405-425.

\bibitem{HT2} H. T. H\`a and A. Van Tuyl, Monomial ideals, edge ideals of hypergraphs, and their
graded Betti numbers, {\it J. Algebraic Combin.}, 27(2) (2008) 215-245.

\bibitem{HTT} H. T.  H\`a, N. V. Trung, and  T. N. Trung, Depth and regularity of powers of sums of ideals, {\it Math. Z.},
282 (3-4) (2016), 819-838.

\bibitem{HT3} Jing He and A. Van Tuyl, Algebraic properties of the path ideal of a tree,  {\it Comm. Algebra}, 38 (5) (2010),
1725-1742.


\bibitem{HH2} J. Herzog and T. Hibi,  {\it  Monomial Ideals}, New York, NY, USA: Springer-Verlag, 2011.

\bibitem{HH3} J. Herzog and T. Hibi,  The depth of powers of an ideal {\it J. Algebra}, 291 (2005), 534-550.



\bibitem{HT4} L. T. Hoa and N. D. Tam, On some invariants of a mixed product of ideals, {\it Arch. Math}, 94 (4), (2010),
327-337.

\bibitem{J}  S.  Jacques, Betti numbers of graph ideals, PhD dissertation, University of
Sheffield, 2004.




\bibitem{JG} J. B. Jensen and G. Gutin, {\it Digraphs. Theory, Algorithms and Applications}, Springer Monographs in
Mathematics, Springer, 2006.


\bibitem{JNS} A. V. Jayanthan, N. Narayanan and S. Selvaraja, Regularity of powers of bipartite graphs, {\it J.
Algebraic Combin.}, 47 (1) (2018), 17-38.



\bibitem{KM} D. Kiani and S. S. Madani, Betti numbers of path ideals of trees, {\it Comm. Algebra}, 44 (12) (2016),
5376-5394.

\bibitem{K} V. Kodiyalam, Asymptotic behaviour of Castelnuovo-Mumford regularity,
{\it Proc. Amer. Math. Soc.}, 128 (1999), 407-411.

\bibitem{LM}Kuei-Nuan Lin and J. McCullough, Hypergraphs and regularity of square-free monomial ideals,
{\it Internat. J. Algebra Comput.}, 23 (7) (2013), 1573-1590.


\bibitem{MPV} J. Mart\'inez-Bernal, Y. Pitones and R. H. Villarreal, Minimum distance functions of graded
ideals and Reed-Muller-type codes, {\it J. Pure Appl. Algebra}, 221 (2017), 251-275.


\bibitem{MSY} M. Moghimian, S. A. Fakhari and S. Yassemi,  Regularity of powers of edge Ideal of whiskered cycles, {\it Comm. Algebra}, 45(3) (2016), 1246-1259.


\bibitem{M}  S. Morey, Depths of powers of the edge ideal of a tree, {\it Comm. Algebra}, 38, (2010),  4042-4055.



\bibitem{MV} S. Morey and R. H. Villarreal, {\it Edge ideals: algebraic and combinatorial properties}, in Progress
in Commutative Algebra, Combinatorics and Homology, Vol. 1 (C. Francisco, L. C. Klingler,
S. Sather-Wagstaff, and J. C. Vassilev, Eds.), De Gruyter, Berlin, 2012, 85-126.


\bibitem{PS} C. Paulsen and S. Sather-Wagstaff, Edge ideals of weighted graphs, {\it J. Algebra Appl.}, 12 (5)
(2013),  1250223-1-24.

\bibitem{PRT} Y. Pitones, E. Reyes, and J. Toledo, Monomial ideals of weighted oriented graphs,
 arXiv:1710.03785.

\bibitem{RJNP}  Y. C. Ruiz, S. Jafari,  N. Nemati and B. Picone,  Regularity of bicyclic Graphs and their powers,
{\it ArXiv: 1802.07202V1}.


\bibitem{TW} N. V. Trung and  H. Wang, On the asymptotic behavior of Castelnuovo-Mumford regularity, {\it J. Pure Appl.
Algebra}, 201 (2005), 42¨C48.


\bibitem{V2} Kodiyalam, Vijay, Asymptotic behaviour of Castelnuovo-Mumford regularity, {\it Proc. Am. Math. Soc.},
128 (2) (2000), 407-411.

\bibitem{W}  R. Woodroofe,  Matchings, coverings, and Castelnuovo-Mumford regularity, {\it J. Commut. Algebra}
6 (2) (2014), 287¨C304.

\bibitem{Z1} Xinxian  Zheng,  Resolutions of facet ideals, {\it Comm. Algebra}, 32 (2004), 2301-2324.



\bibitem{Z2} Guangjun Zhu, Projective dimension and regularity of the path ideal of the line graph,  {\it J. Algebra Appl.}, 17 (4), (2018), 1850068-1-15.

\bibitem{Z3} Guangjun Zhu, Projective dimension and  regularity of  path  ideals  of  cycles, {\it J. Algebra Appl.}, 17 (10), (2018), 1850188-1-22.

\bibitem{Z4} Guangjun Zhu, Li Xu, Hong Wang and Zhongming Tang, Projective dimension and regularity of edge ideal of some weighted oriented graphs,  To appear in {\it Rocky MT J. Math.}.


\end{thebibliography}
\end{document}